\documentclass[12pt]{amsart}
\usepackage{amssymb}
\usepackage{enumerate}
\usepackage{graphicx}
\makeatletter
\@namedef{subjclassname@2010}{%
  \textup{2010} Mathematics Subject Classification}
\makeatother
\newtheorem{thm}{Theorem}[section]
\newtheorem{cor}[thm]{Corollary}
\newtheorem{lem}[thm]{Lemma}
\newtheorem{prop}[thm]{Proposition}

\theoremstyle{definition}
\newtheorem{defn}[thm]{Definition}
\newtheorem{rem}[thm]{Remark}

\numberwithin{equation}{section}

\begin{document}
\normalsize
\baselineskip=30pt

\title[ infinite products using multiplicative modulus function ]{N\lowercase{ovel} A\lowercase{pproach} \lowercase{to} I\lowercase{nfinite} P\lowercase{roducts} U\lowercase{sing} M\lowercase{ultiplicative} M\lowercase{odulus} F\lowercase{unction}}

\author{C. G\lowercase{anesa} M\lowercase{oorthy}}
\address{Department of Mathematics\\ Alagappa University\\
Karaikudi-630 004, India}
\email{ganesamoorthyc@gmail.com}

\date{}

\begin{abstract}
\ The usual nonnegative modulus function is based on addition. A natural different modulus function on the set of positive reals is introduced. Arguments for results for series through the usual modulus function are transformed to arguments for results for infinite products through the new modular function for multiplication. Counterparts for Riemann rearrangement theorem and some tests for convergence are derived. These counterparts are completely new results and they are different from classical results for infinite products.
\end{abstract}


\maketitle

\section{Introduction}
\ After giving the definition of convergence of infinite products of complex numbers, the following two results can be derived (see, e.g., \cite{a1}).
$(\textrm{i})$ For a sequence $(a_{n})_{n=1}^{\infty}$ of positive reals, $\prod\limits_{n=1}^{\infty} (1+a_{n})$ converges if and only if $\sum\limits_{n=1}^{\infty}a_{n} $ converges, and if and only if $\prod\limits_{n=1}^{\infty} (1- a_{n})$ converges. $(\textrm{ii})$ For a sequence $(a_{n})_{n=1}^{\infty}$ of complex numbers, $\prod\limits_{n=1}^{\infty} (1+a_{n})$ converges whenever $\prod\limits_{n=1}^{\infty} (1+|a_{n}|)$ converges. The convergence of $\prod\limits_{n=1}^{\infty} (1+|a_{n}|)$ in the second result is called absolute convergence of $\prod\limits_{n=1}^{\infty} (1+a_{n})$. However, our expectation is that the absolute convergence of $\prod\limits_{n=1}^{\infty} (1+a_{n})$ should mean the convergence of $\prod\limits_{n=1}^{\infty} |1+a_{n}|$. This expectation is to be fulfilled by introducing a natural modulus function for multiplication. The results to be obtained for infinite products use transformed arguments of arguments used for results for infinite series. A counterpart of Riemann's rearrangement theorem is to be derived. Counterparts of tests for convergence of series are to be derived for infinite products. Tests for convergence provide necessary conditions and sufficient conditions for convergence.

\ There is a natural procedure to convert problems in infinite products to problems in infinite series. This procedure is the one which uses the exponential function and the natural logarithmic function for transformation. For example, convergence of $\prod\limits_{n=1}^{\infty} a_{n}$ is equivalent to convergence of $\sum\limits_{n=1}^{\infty} \log a_{n}$, and convergence of $\sum\limits_{n=1}^{\infty} a_{n}$ is equivalent to convergence of $\prod\limits_{n=1}^{\infty} e^{a_{n}}$, formally. This natural procedure is not to be applied directly in deriving results for infinite products. However this transformation is to be applied indirectly in arguments for results to be derived. Only infinite products with positive factors will be considered. So, the following simplified definition will be adopted.

\begin{defn}
\ Let $(a_{n})_{n=1}^{\infty}$ be a sequence of positive reals. Let $ p_{n}=\prod\limits_{k=1}^{n}a_{k} $, for every $n=1,2,3,...$.If $p_{n}\rightarrow p$ for some real $p>0$ as $ n \rightarrow \infty $, then it is said that the infinite product $\prod\limits_{n=1}^{\infty} a_{n}$ converges to $p$ and it is written as $p=\prod\limits_{n=1}^{\infty} a_{n}$. Otherwise, it is said that $\prod\limits_{n=1}^{\infty} a_{n}$ does not converge.
\end{defn}

\begin{rem}
\ The cases $p=0$ and $p=+\infty$ are not included for convergence. A necessary condition for convergence is $a_{n} \rightarrow 1$ as $n\rightarrow\infty$, but this is not sufficient. A necessary and sufficient condition is given in the following known theorem. See \cite{a1}.
\end{rem}

\begin{thm}
\ An infinite product $\prod\limits_{n=1}^{\infty} a_{n}$, with $a_{n}>0$, $\forall n$, converges if and only if for every $\epsilon >0$, there exists an integer $n_{0}$ such that $|a_{m}a_{m+1}...a_{n}-1|<\epsilon$, whenever $n>m\geq n_{0}$.
\end{thm}

\begin{rem}
\ Let $1<p(1)<p(2)<$... be a sequence of integers such that $p(n+1)-p(n)\leq M$, for some finite $M>0$. If $a_{1}a_{2}...a_{p(1)}a_{p(2)+1} a_{p(2)+2}...a_{p(3)}a_{p(4)+1}...a_{p(2k)+1}...a_{p(2k+1)}$... and $a_{p(1)+1}a_{p(1)+2}...a_{p(2)}a_{p(3)+1}a_{p(3)+2}...a_{p(4)}a_{p(5)+1}...a_{p(2k+1)+1}...a_{p(2k+2)}$...converge to $r$ and $s$, then $a_{1}a_{2}$... converges to $rs$. This result can be derived by using Theorem 1.3. This can be done by modifying the arguments of Theorem 8.14 in \cite{a1}. For example, if $(b_{n})_{n=1}^{\infty}$ and $(c_{n})_{n=1}^{\infty}$ are sequences of positive numbers such that $\prod\limits_{n=1}^{\infty} b_{n}$ and $\prod\limits_{n=1}^{\infty} c_{n}$ converge, then $b_{1}c_{1}b_{2}c_{2}b_{3}c_{3}$... and $b_{1}c_{1}^{-1}b_{2}c_{2}^{-1}b_{3}c_{3}^{-1}$... converge.
\end{rem}

\section{Modulus Function for Multiplication:}
\ For a given real number $x$, let
\begin{equation*}
|x|_{+}=max \{x,-x \}= \begin{cases}
\ x & \text{if $ x\geq 0$}\\
 -x & \text{if $ x\leq 0 $}.
\end{cases}
\end{equation*}
\ Here $|x|_{+}$ is the usual absolute value $|x|$. Let us call $|$ $|_{+}$ as additive absolute value function (or additive modulus function). Observe that $0$ is the additive identity in the additive group of real numbers and $-x$ is the additive inverse of $x$. Observe also that $1$ is the multiplicative identity in the multiplicative group of positive numbers, and $x^{-1}$ is the multiplicative inverse of $x>0$.

\begin{defn}
\ For a given positive real number $x$, let us define
\begin{equation*}
|x|_{\times}=max \{x,x^{-1} \}= \begin{cases}
\ x & \text{if $ x\geq 1 $}\\
 x^{-1} & \text{if $ x\leq 1 $}.
\end{cases}
\end{equation*}
\ Let us call $|$ $ |_\times$ as multiplicative absolute value function (or multiplicative modulus function). Here $|x|_\times ^{-1} \leq x \leq |x|_\times$. Also, $|xy|_\times \leq |x|_\times |y|_\times$ for $x>0$ and $y>0$.
\end{defn}

\ Positive and negative parts corresponding to multiplication can also be defined with an understanding that numbers in $(0,1)$ are multiplicative negative numbers and that number in $(1,+\infty)$ are multiplicative positive numbers.

\begin{lem}
\ For a given $x>0$, let $p=(|x|_\times x)^{1/2}$, $q=\Big(\frac{|x|_\times}{x}\Big)^{1/2}$ . Then $\frac{p}{q}=x$, $pq=|x|_\times$, $p\geq 1$ and $q \geq 1$.
\end{lem}
\begin{proof}
\ Direct verification.
\end{proof}

\ The numbers $p$ and $q$ may be considered as multiplicative positive part and multiplicative negative part of $x$ in Lemma 2.2.

\begin{defn}
\ Let $(a_{n})_{n=1}^{\infty}$ be a sequence of positive numbers. Then $\prod\limits_{n=1}^{\infty} a_{n}$ is said to converge  m-absolutely, if $\prod\limits_{n=1}^{\infty} |a_{n}|_\times$ converges.
\end{defn}

\begin{lem}
\ Let $(a_{n})_{n=1}^{\infty}$ be a sequence of positive numbers such that $\prod\limits_{n=1}^{\infty} |a_{n}|_\times$ converges. Then  $\prod\limits_{n=1}^{\infty} a_{n}$ converges and $\Big(\prod\limits_{n=1}^{\infty} |a_{n}|_\times\Big)^{-1} \leq \prod\limits_{n=1}^{\infty} a_{n} \leq \prod\limits_{n=1}^{\infty} |a_{n}|_\times$.
\end{lem}
\begin{proof}
\ Note that for $m>n$,
\begin{eqnarray*}
 \prod\limits_{k=n}^{m} \frac{1}{|a_{k}|_\times} \leq \prod\limits_{k=n}^{m} a_{k} \leq \prod\limits_{k=n}^{m} |a_{k}|_\times.
\end{eqnarray*}
By Theorem 1.3, $\prod\limits_{n=1}^{\infty} a_{n}$ converges and
\begin{eqnarray*}
0<\Big(\prod\limits_{n=1}^{\infty} |a_{n}|_\times\Big)^{-1} \leq \prod\limits_{n=1}^{\infty} a_{n} \leq \prod\limits_{n=1}^{\infty} |a_{n}|_\times < \infty.
\end{eqnarray*}
\end{proof}
\ Lemma 2.4 may be roughly stated as ``$m$-absolute convergence implies convergence".

\begin{lem}
\ Let $(a_{n})_{n=1}^{\infty}$ be a sequence of positive reals converging to a positive number $K$. Let $(t_{n})_{n=1}^{\infty}$ be a sequence of positive reals such that $\sum\limits_{n=1}^{\infty} t_{n}= +\infty$. Then $(a_{1}^{t_{1}}a_{2}^{t_{2}}...a_{n}^{t_{n}})^\frac{1}{t_{1}+t_{2}+...+t_{n}} \rightarrow K$ as $n \rightarrow \infty$.
\end{lem}
\begin{proof}
\ Without loss of generality, let us assume that $K=1$. Let $\epsilon >0$ be given. Find an integer $k$ such that $|a_{n}|_\times <(1+\epsilon),$ $\forall n \geq k$. Find $M >0$ such that $|a_{n}|_\times \leq M < \infty$, $\forall n$. Then
\begin{eqnarray*}
1 \leq (a_{1}^{t_{1}}a_{2}^{t_{2}}...a_{n}^{t_{n}})^\frac{1}{t_{1}+t_{2}+...+t_{n}} \leq (|a_{1}|_\times^{t_{1}}...|a_{k}|_\times^{t_{k}})^\frac{1}{t_{1}+t_{2}+...+t_{n}} (|a_{k+1}|_\times^{t_{k+1}}...|a_{n}|_\times^{t_{n}})^\frac{1}{t_{1}+t_{2}+...+t_{n}}\\
\leq M^\frac{t_{1}+...+t_{k}}{t_{1}+t_{2}+...+t_{n}} (1+\epsilon)^\frac{t_{k+1}+...+t_{n}}{t_{1}+t_{2}+...+t_{n}}.
\end{eqnarray*}
Then, the right hand side tends to $1+\epsilon$ as $n \rightarrow \infty$, for every $\epsilon >0$. Thus, $a_{1}^{t_{1}}a_{2}^{t_{2}}...a_{n}^{t_{n}} \rightarrow 1 $ as $n \rightarrow \infty$. This proves the result.
\end{proof}

\section{Rearrangements}
\ The first Theorem 3.1 is a transformed version of the classical Riemann rearrangement theorem. Let us use the word ``rearrangement" in the usual sense.
\begin{thm}
\ Let $(a_{n})_{n=1}^{\infty}$ be a sequence of positive real numbers such that $\prod\limits_{n=1}^{\infty} a_{n}$ converges and $\prod\limits_{n=1}^{\infty} |a_{n}|_\times$ does not converge. Suppose $0 \leq \alpha \leq \beta \leq \infty$. Then there exists a rearrangement $\prod\limits_{n=1}^{\infty} a_{n}'$ with partial products $u_{n}'=\prod\limits_{k=1}^{n} a_{k}'$ such that $\liminf\limits_{n\rightarrow \infty} u_{n}'=\alpha$ and $\limsup\limits_{n \rightarrow \infty} u_{n}'=\beta $.
\end{thm}

\begin{proof}
\ Let $p_{n}=(|a_{n}|_\times a_{n})^{1/2}$ and $q_{n}=\Big(\frac{|a_{n}|_\times}{a_{n}}\Big)^{1/2}$, for every $n$. By Lemma 2.2, $\frac{p_{n}}{q_{n}}=a_{n}, p_{n}q_{n}=|a_{n}|_\times, p_{n} \geq 1$ and $q_{n} \geq 1$, for every $n$. Since $\prod\limits_{n=1}^{\infty} (p_{n}q_{n})=\prod\limits_{n=1}^{\infty}|a_{n}|_\times$, either $\prod\limits_{k=1}^{n} p_{k}\rightarrow +\infty$ as $n\rightarrow \infty$
or $\prod\limits_{k=1}^{n} q_{k}\rightarrow \infty$ as $n\rightarrow \infty$. Since $\frac{\prod\limits_{k=1}^{n} p_{k}}{\prod\limits_{k=1}^{n} q_{k}}=\prod\limits_{k=1}^{n} a_{k}$ converges as $n\rightarrow \infty$, if $\prod\limits_{k=1}^{n} p_{k}$ converges as $n\rightarrow \infty$, then $\prod\limits_{k=1}^{n} q_{k}$ converges as $n\rightarrow \infty$, and similarly, if $\prod\limits_{k=1}^{n} q_{k}$ converges as $n\rightarrow \infty$, then $\prod\limits_{k=1}^{n} p_{k}$ converges as $n\rightarrow \infty$. Hence $\prod\limits_{k=1}^{n} p_{k} \rightarrow \infty$ and $\prod\limits_{k=1}^{n} q_{k} \rightarrow \infty$ as $n\rightarrow \infty$. Now, let $p_{1}',p_{2}',...$ denote the factors of $\prod\limits_{n=1}^{\infty} a_{n}$  which are greater than or equal to $1$ in the order in which they occur, and let $q_{1}',q_{2}',...$  be the multiplicative absolute values of the remaining factors of $\prod\limits_{n=1}^{\infty} a_{n}$ in their original order. The products $\prod\limits_{n=1}^{\infty} p_{n}'$  and $\prod\limits_{n=1}^{\infty} q_{n}'$  are different from $\prod\limits_{n=1}^{\infty} p_{n}$ and $\prod\limits_{n=1}^{\infty} q_{n}$  only by factors $1$, and therefore $\prod\limits_{k=1}^{n} p_{k}' \rightarrow \infty$ and $\prod\limits_{k=1}^{n} q_{k}' \rightarrow \infty$ as $n\rightarrow \infty$. Choose sequences
$(\alpha_{n})_{n=1}^{\infty}$ and $(\beta_{n})_{n=1}^{\infty}$ of positive numbers such that $\alpha_{n}\rightarrow\alpha$ and $\beta_{n}\rightarrow\beta$ as $n\rightarrow\infty, \alpha_{n} < \beta_{n}$, $\forall n$, and $\beta_{1} \geq 1$. Let $m_{1}, k_{1}$ be the smallest integers such that $p_{1}'p_{2}'...p_{m_{1}}' > \beta_{1}$, $p_{1}'p_{2}'...p_{m_{1}}' q_{1}'^{-1}q_{2}'^{-1}...q_{k_{1}}'^{-1} < \alpha_{1}$ ; let $m_{2}, k_{2}$ be the smallest integers such that $p_{1}'p_{2}'...p_{m_{1}}' q_{1}'^{-1}q_{2}'^{-1}...q_{k_{1}}'^{-1} p_{m_{1}+1}'...p_{m_{2}}'> \beta_{2}$, $p_{1}'p_{2}'...p_{m_{1}}' q_{1}'^{-1}q_{2}'^{-1}...q_{k_{1}}'^{-1} p_{m_{1}+1}'...p_{m_{2}}' q_{k_{1}+1}'^{-1}...q_{k_{2}}'^{-1}< \alpha_{2} $;  and let us continue in this way. Let $x_{n},y_{n}$ denote the partial products of the rearrangement $p_{1}'p_{2}'...p_{m_{1}}' q_{1}'^{-1}q_{2}'^{-1}...q_{k_{1}}'^{-1} p_{m_{1}+1}'...p_{m_{2}}' q_{k_{1}+1}'^{-1}...q_{k_{2}}'^{-1}p_{m_{2}+1}'...$ whose last factors are $p_{m_{n}}', q_{k_{n}}'^{-1}$. Then $\frac{x_{n}}{p_{m_{n}}'} \leq \beta_{n}$ and $\frac{y_{n}}{q_{k_{n}}'^{-1}} \geq \alpha_{n}$ so that $p_{m_{n}}'^{-1} \leq \frac{\beta_{n}}{x_{n}} \leq 1$ and $1 \leq \frac{\alpha_{n}}{y_{n}} \leq q_{k_{n}}'$. Since $p_{n}' \rightarrow1$ and $q_{n}'\rightarrow1$ as $n \rightarrow\infty$, then $x_{n}\rightarrow \beta$ and $y_{n} \rightarrow\alpha $ as $n\rightarrow\infty$. From our construction, if $u_{n}'$ is the $n$-th partial product of the constructed rearrangement, then it is clear that $\liminf\limits_{n\rightarrow \infty} u_{n}'=\alpha$ and $\limsup\limits_{n \rightarrow \infty} u_{n}'=\beta $.
 \end{proof}

\begin{rem}
\ It should be observed from this Theorem 3.1 that if every rearrangement converges then $m$-absolute convergence holds.
\end{rem}
\begin{thm}
\ Let $(a_{n})_{n=1}^{\infty}$ be a sequence of positive numbers such that $\prod\limits_{n=1}^{\infty} |a_{n}|_\times$ converges. Then every rearrangement of $\prod\limits_{n=1}^{\infty} a_{n}$ converges and all of them converge to the same sum.
\end{thm}

\begin{proof}
\ Let $M=\prod\limits_{n=1}^{\infty} |a_{n}|_\times$. Let $u_{n}=\prod\limits_{k=1}^{n} a_{k}$ and $u_{n}'=\prod\limits_{k=1}^{n} a_{k}'$ be partial products of $\prod\limits_{k=1}^{\infty} a_{k}$ and its one rearrangement $\prod\limits_{k=1}^{\infty} a_{k}'$. To each $n$, let $A_{n}=\Big\{ k \in \{1,2,...,n \}: |a_{k}|_\times=|a_{k}'|_\times \Big\}$ and $B_{n} = \{1,2,...,n \}\setminus A_{n}$. Then
\begin{eqnarray*}
 0 \leq \Big| \prod\limits_{k=1}^{n} |a_{k}|_\times - \prod\limits_{k=1}^{n} |a_{k}'|_\times \Big|_{+} \leq M \Big| \prod\limits_{k \in B_{n}} |a_{k}|_\times - \prod\limits_{k \in B_{n}} |a_{k}'|_\times \Big|_{+}\\
  \leq M \Bigg[ \Big| 1- \prod\limits_{k \in B_{n}} |a_{k}|_\times \Big|_{+} + \Big|1 - \prod\limits_{k \in B_{n}} |a_{k}'|_\times \Big|_{+} \Bigg].
\end{eqnarray*}
\ By Theorem 1.3, the right hand side tends to zero as $n \longrightarrow \infty$, because $|a_{k}|_\times \geq 1$ and $|a_{k}'|_\times \geq 1$, $\forall k$. Therefore $\prod\limits_{k=1}^{\infty} |a_{k}|_\times = \prod\limits_{k=1}^{\infty} |a_{k}'|_\times$, and
$\prod\limits_{k=1}^{\infty} |a_{k}'|_\times $ converges. By Lemma 2.4, $\prod\limits_{k=1}^{\infty} a_{k}'$ converges. In the previous arguments, for $B_{n}$ defined above, let us observe that $\Big| 1- \prod\limits_{k \in B_{n}} |a_{k}|_\times^{-1} \Big|_{+} + \Big|1 - \prod\limits_{k \in B_{n}} |a_{k}'|_\times \Big|_{+}$ tends to zero as $n \rightarrow \infty$. This is applicable even for subsets of $B_{n}$. So, if $C_{n}= \Big\{ k \in \{1,2,...,n \}: a_{k}=a_{k}' \Big\}$ and $D_{n}=\{1,2,...,n \}\setminus C_{n}$, then $\Big| 1- \prod\limits_{k \in D_{n}} a_{k} \Big|_{+} + \Big|1 - \prod\limits_{k \in D_{n}} a_{k}' \Big|_{+} \rightarrow 0$ as $n \rightarrow \infty$, because $\prod\limits_{k \in D_{n}} |a_{k}|_\times^{-1} \leq \prod\limits_{k \in D_{n}} a_{k} \leq \prod\limits_{k \in D_{n}} |a_{k}|_\times $ and $\prod\limits_{k \in D_{n}} |a_{k}'|_\times^{-1} \leq \prod\limits_{k \in D_{n}} a_{k}' \leq \prod\limits_{k \in D_{n}} |a_{k}'|_\times $. Thus
\begin{eqnarray*}
 0 \leq \Big| \prod\limits_{k=1}^{n} a_{k} - \prod\limits_{k=1}^{n} a_{k}' \Big|_{+} \leq M \Bigg[ \Big| 1- \prod\limits_{k \in D_{n}} a_{k} \Big|_{+} + \Big|1 - \prod\limits_{k \in D_{n}} a_{k}' \Big|_{+} \Bigg]
\end{eqnarray*}
implies that $\prod\limits_{k=1}^{\infty} a_{k} = \prod\limits_{k=1}^{\infty} a_{k}'$.
\end{proof}

\begin{defn}
\ A product $\prod\limits_{k=1}^{\infty} b_{k}$ is called a subproduct of $\prod\limits_{k=1}^{\infty} a_{k}$ if $b_{k}=a_{k}$ or $1$ for all $k$, and $b_{k}=a_{k}$ for infinitely many $k$.
\end{defn}
\begin{thm}
\ Let $(a_{n})_{n=1}^{\infty}$ be a sequence of positive numbers such that $\prod\limits_{n=1}^{\infty} |a_{n}|_{\times}$ converges. Consider a class of all products of the form $\prod\limits_{n=1}^{\infty} a_{n}^{c_{n}}$, where $-1 \leq c_{n} \leq 1$, $\forall n$. Then all these products converge uniformly in the following sense. For every $\epsilon > 0$, there is an integer $n_{0}$, which is common for all sequences $(c_{n})_{n=1}^{\infty}$ satisfying $-1 \leq c_{n} \leq 1$, $\forall n$, such that $\Big|1- \prod\limits_{k=n}^{m} a_{k}^{c_{k}}\Big|_{+} <\epsilon $, $\forall m, n \geq n_{0}$ satisfying $m>n$, and so that $\Big|1- \prod\limits_{k=n}^{\infty} a_{k}^{c_{k}}\Big|_{+} \leq \epsilon $, $\forall n \geq n_{0}$.
\end{thm}
\begin{proof}
 $\prod\limits_{k=n}^{m} |a_{k}|_{\times}^{-1} \leq \prod\limits_{k=n}^{m} a_{k}^{c_{k}} \leq \prod\limits_{k=n}^{m} |a_{k}|_{\times}$, when $-1 \leq c_{k} \leq 1$, $\forall k$, and when $m>n$. Thus
 \begin{eqnarray*}
 \Big|1- \prod\limits_{k=n}^{m} a_{k}^{c_{k}}\Big|_{+} \leq \max \Bigg\{ \Big|1- \prod\limits_{k=n}^{m} |a_{k}|_{\times}^{-1}\Big|_{+}, \Big|1- \prod\limits_{k=n}^{m} |a_{k}|_{\times}\Big|_{+} \Bigg\}.
 \end{eqnarray*}
 This observation proves the result, when Theorem 1.3 is applied.
\end{proof}
\begin{cor}
\ Let $(a_{n})_{n=1}^{\infty}$ be a sequence of positive numbers such that $\prod\limits_{n=1}^{\infty} |a_{n}|_{\times}$ converges. Then, all subproducts $\prod\limits_{k=1}^{\infty} b_{k}$ of $\prod\limits_{k=1}^{\infty} a_{k}$ converge uniformly in the following sense. For every $\epsilon >0$, there is an integer $n_{0}$, which is common for all subproducts $\prod\limits_{k=1}^{\infty} b_{k}$, such that $\Big|1- \prod\limits_{k=n}^{m} b_{k}\Big|_{+} <\epsilon $, $\forall m, n \geq n_{0}$ satisfying $m>n$, and so that $\Big|1- \prod\limits_{k=n}^{\infty} b_{k}\Big|_{+} \leq \epsilon $, $\forall n \geq n_{0}$.
\end{cor}

\ Theorem 3.5 is placed in this section, because its proof uses arguments which are used in the proof of Theorem 3.3. The next section is also devoted to rearrangements in terms of unordered products.
\section{Unordered Products:}
\ Let $I$ be an infinite set. In this section, $E,E_{1},E_{2},F,F_{1},F_{2}$ are notations to be used for finite subsets of $I$ with or without mentioning finiteness of them. Let $(a_{i})_{i \in I}$ be a collection of positive numbers. Let $I_{1}=\{i \in I: a_{i}>1 \}$, $I_{2}=\{i \in I: a_{i}<1 \}$ and $I_{3}=\{i \in I: a_{i}=1 \}$. These notations are also fixed in this section. Let us say that $\prod\limits_{i \in I} a_{i}$ converges, if there is a number $p \in (0,\infty)$ such that for given $\epsilon >0$, there is a finite set $F \subseteq I$ such that $\Big|\prod\limits_{i \in E} a_{i} - p\Big|_{+} < \epsilon $, for every finite set $E$ satisfying $F\subseteq E \subseteq I$. In this case, let us say that $\prod\limits_{i \in I} a_{i}$ converges to $p$ and let us write $\prod\limits_{i \in I} a_{i}=p$. Equivalently, this convergence happens when and only when for given $\epsilon >0$, there is a set $F \subseteq I$ such that $\Big|\frac{\prod\limits_{i \in E} a_{i}}{p} - 1\Big|_{+} < \epsilon $, for every $E$ satisfying $F\subseteq E \subseteq I$. One may find arguments, which are necessary to prove Theorem 1.3, in the proof of the next lemma.
\begin{lem}
\ $\prod\limits_{i \in I} a_{i}$ converges if and only if for given $ \epsilon >0$, there is a finite set $F \subseteq I$ such that $\Big|\prod\limits_{i \in E} a_{i} - 1\Big|_{+} < \epsilon $, whenever $ E \subseteq I\setminus F $.
\end{lem}
\begin{proof}
\ Suppose $\prod\limits_{i \in I} a_{i}$ converges to some $p >0$. Find $F_{1} \subseteq I$ such that $\Big|\prod\limits_{i \in E} a_{i} - p\Big|_{+} < \frac{p}{4} $, whenever $ F_{1} \subseteq E \subseteq I $, and so that $0<\frac{3}{4}p<\prod\limits_{i \in E} a_{i}<\frac{5}{4}p$. For given $\epsilon>0$, there is a set $F_{2}\supseteq F_{1}$ such that $\Big| \prod\limits_{i \in E} a_{i}-\prod\limits_{i \in F_{2}} a_{i}\Big|_{+}<\frac{3}{4}p\epsilon$ whenever $E\supseteq F_{2}$, and so that $\Big| \prod\limits_{i \in E\setminus F_{2}} a_{i}-1\Big|_{+}<\frac{3p\epsilon}{4\prod\limits_{i \in F_{2}} a_{i}}<\epsilon$. That is, $\Big| \prod\limits_{i \in E} a_{i}-1\Big|_{+}<\epsilon$, whenever $E\subseteq I \setminus F_{2}$.

\ Conversely, assume that for given $\epsilon > 0$, there is a set $F$ such that $\Big| \prod\limits_{i \in E} a_{i}-1\Big|_{+}<\epsilon$, whenever $E\subseteq I \setminus F$. In particular, there is a set $F_{1}$ such that $\Big| \prod\limits_{i \in E} a_{i}-1\Big|_{+}<\frac{1}{4}$, whenever $E\subseteq I \setminus F_{1}$, and so that  $\frac{3}{4}<\prod\limits_{i \in E} a_{i}<\frac{5}{4}$. Therefore, $\frac{3}{4}M<\prod\limits_{i \in E} a_{i}<\frac{5}{4}M$, for any $E\subseteq I$, where $M=\prod\limits_{i \in F_{1}} a_{i}$. Fix $\epsilon>0$. Find a set $F_{2}\supseteq F_{1}$ such that $\Big| \prod\limits_{i \in E} a_{i}-1 \Big|_{+}<\frac{3M\epsilon}{4}$, whenever $E\subseteq I \setminus F_{2}$. Then, for $ F_{2} \subseteq E_{2}\subseteq E_{1}$, $\Big|\prod\limits_{i \in E_{1}} a_{i}-\prod\limits_{i \in E_{2}} a_{i}\Big|_{+}=\Big| \prod\limits_{i \in E_{1}\setminus E_{2}} a_{i}-1\Big|_{+}\Big| \prod\limits_{i\in E_{2}}a_{i}\Big|^{-1}_{+}<\epsilon$. Therefore, by triangle inequality, $\Big|\prod\limits_{i \in E_{1}} a_{i}-\prod\limits_{i \in E_{2}} a_{i}\Big|_{+}<2\epsilon$, whenever $F_{2}\subseteq E_{1}$ and $F_{2}\subseteq E_{2}$. Hence $\prod\limits_{i\in I}a_{i}$ converges.
\end{proof}
\ The following inequality will be used at the end of the proof of the next lemma. $|1-ab|_{+} \leq |1-a|_{+}+|a|_{+}|1-b|_{+}$.
\begin{lem} $\prod\limits_{i\in I} a_{i}$ converges if and only if $\prod\limits_{i\in I_{1}} a_{i}$ converges and $\prod\limits_{i\in I_{2}} a_{i}$ converges. (Convention: Finite products converge). Moreover, $\prod\limits_{i\in I} a_{i}=\Big(\prod\limits_{i\in I_{1}} a_{i}\Big)\Big(\prod\limits_{i\in I_{2}} a_{i}\Big)$, in this case.
\begin{proof}
Suppose $\prod\limits_{i\in I}a_{i}$ converges. Fix $\epsilon>0$. Then there is a set $F$ such that $\Big|\prod\limits_{i\in E}a_{i}-1\Big|_{+}<\epsilon$, whenever $E\subseteq I\setminus F$. In particular,  $\Big|\prod\limits_{i\in E}a_{i}-1\Big|_{+}<\epsilon$, whenever $E\subseteq I_{1}\setminus (I_{1}\cap F)$, and $\Big|\prod\limits_{i\in E}a_{i}-1\Big|_{+}<\epsilon$, whenever $E\subseteq I_{2}\setminus (I_{2}\cap F)$. This proves that $\prod\limits_{i\in I_{1}} a_{i}$ and $\prod\limits_{i\in I_{2}} a_{i}$ converge. Since $\prod\limits_{i\in E} a_{i}=\Big(\prod\limits_{i\in E \cap I_{1}} a_{i}\Big)\Big(\prod\limits_{i\in E \cap I_{2}} a_{i}\Big)$ for any $E$ and multiplication is continuous in the real line, then $\prod\limits_{i\in I} a_{i}=\Big(\prod\limits_{i\in I_{1}} a_{i}\Big)\Big(\prod\limits_{i\in I_{2}} a_{i}\Big)$. This continuity of multiplication should be observed in terms of nets over directed sets $I_1,I_2$, and $I$.

Conversely assume that $\prod\limits_{i\in I_{1}}a_{i}$ and $\prod\limits_{i\in I_{2}}a_{i}$ converge. Fix $\epsilon>0$. There are $F_{1}$, $F_{2}$ such that $\Big|\prod\limits_{i\in E}a_{i}-1\Big|_{+}<\epsilon$, whenever $E\subseteq I_{1}\setminus F_{1}$, and $\Big|\prod\limits_{i\in E}a_{i}-1\Big|_{+}<\epsilon$, whenever $E\subseteq I_{2}\setminus F_{2}$. Then $\Big|\prod\limits_{i\in E}a_{i}-1\Big|_{+}<\epsilon+(1+\epsilon)\epsilon$, whenever $E\subseteq I\setminus F_{1}\cup F_{2}$. This proves that $\prod\limits_{i\in I}a_{i}$ converges.
\end{proof}
\end{lem}
\begin{rem}
Suppose $\prod\limits_{i\in I}a_{i}$ converges. For each $n$, let $G_{n}=\{i\in I: a_{i}<1-\frac{1}{n}\}$ and $H_{n}=\{i\in I: a_{i}>1+\frac{1}{n}\}$.  Then $G_{n}$ and $H_{n}$ are finite sets, because $\prod\limits_{i\in I_{1}}a_{i}$ and $\prod\limits_{i\in I_{2}}a_{i}$ converge. So, $I_{1}=\bigcup\limits^{\infty}_{n=1} H_{n}$ and $I_{2}=\bigcup\limits^{\infty}_{n=1} G_{n}$ are countable sets. Thus, let us replace $I$ by the countable set $I_{1} \cup I_{2}$, when there is a need for simplification in notation in connection with convergence. On some occasions, it will be fixed as $I= \{1,2,...\}$. It should be observed that convergence of all these products should be explained only in terms of nets over directed sets $I_1,I_2$, and $I$.
\end{rem}
\begin{thm}
\ The followings are equivalent. \\
\ (a) $\prod\limits_{i\in I} |a_{i}|_{\times} $ converges. \\
\ (b) $\prod\limits_{i\in I}a_{i}$ converges. \\
\ (c)  Let $D$ be a directed set which is a cofinal subset of the directed set $\{ E: E$ is a finite subset of $I \}$ under the inclusion relation. That is, if $E$ is any finite subset of $I$, then there is a set $F \in D$ such that $E \subseteq F$. (See [2] for terminology). Then the net $\Big( \prod\limits_{i\in F} |a_{i}|_{\times}\Big)_{ F \in D} $ converges for any $D$ mentioned above. Moreover, all these nets converge to $\prod\limits_{i\in I} |a_{i}|_{\times} $.\\
\ (d) The net $\Big( \prod\limits_{i\in F} a_{i}\Big)_{F \in D}  $ converges for any $D$ mentioned in (c). Moreover, all these nets converge to $\prod\limits_{i\in I} a_{i} $.\\
\ (e) Consider $I$ in the form $ \{1,2,...\}$ as it was mentioned in the previous Remark 4.3, on assuming (a) or (b). Write $I$ in the form $I= J_{1} \cup J_{2} \cup....$ in which $J_{i}$ are pairwise disjoint infinite sets. Let us write each $J_{i}$ in the ordered form $J_{i} =\{ i_{1}, i_{2},...\}$. Then for each $i$, the product $\prod\limits^{\infty}_{j=1} a_{i_{j}}$ converges, and then the product $\prod\limits^{\infty}_{i=1} \Big(\prod\limits^{\infty}_{j=1} a_{i_{j}}\Big)$ converges, for any decomposition $J_{1} \cup J_{2} \cup....$ of $I$. In this case, they all converge to $\prod\limits_{i\in I}a_{i}$.\\
\ Note: The statement in part (e) is given in this specific form only for the purpose of the implication: (a) or (b) implies (e). For the reverse implication, let us take I as a countable set.
\end{thm}
\begin{proof}
\ Suppose $\prod\limits_{i\in I} |a_{i}|_{\times} $ converges. Then for any finite subset $F$ of $I$,
\begin{eqnarray*}
 \Big(\prod\limits_{i\in F} |a_{i}|_{\times}\Big)^{-1}  \leq \prod\limits_{i\in F} a_{i}  \leq \prod\limits_{i\in F} |a_{i}|_{\times},
\end{eqnarray*}
and Lemma 4.1 implies that $\prod\limits_{i\in I} a_{i} $ converges.

\ Conversely assume that $\prod\limits_{i\in I} a_{i} $ converges. Then $\prod\limits_{i\in I_{1}}a_{i}$ and $\prod\limits_{i\in I_{2}}a_{i}$ converges, by Lemma 4.2. So, $\prod\limits_{i\in I_{1}} |a_{i}|_{\times} $ and $\prod\limits_{i\in I_{2}} |a_{i}|_{\times} $ converges.
 By continuity of multiplication, $\prod\limits_{i\in I} |a_{i}|_{\times} $ converges.

 \ Thus (a) and (b) are equivalent. From the definition for net convergence, (a), (b), (c), (d) are equivalent, and common limits for nets also do exist for (c) and (d).

 \ To prove the equivalence of (a) and (e), suppose that $\prod\limits_{i\in I} |a_{i}|_{\times} $ converges. By Lemma 4.1, and Theorem 1.3, $\prod\limits^{\infty}_{j=1} |a_{i_{j}}|_{\times}$ converges for each $i$, when the notations given in (e) are used. Moreover,
 \begin{eqnarray*}
 1 \leq \prod\limits^{n}_{i=m} \prod\limits^{\infty}_{j=1} |a_{i_{j}}|_{\times}  \leq \prod\limits_{i\in I} |a_{i}|_{\times},
\end{eqnarray*}
for any $n, m$ satisfying $n> m$. Again by Lemma 4.1 and Theorem 1.3, it can be observed that $\prod\limits^{\infty}_{i=1} \Big(\prod\limits^{\infty}_{j=1} |a_{i_{j}}|_{\times}\Big)$ converges. The inequalities, for $n>m$,
\begin{eqnarray*}
 \Big(\prod\limits^{n}_{j=m} |a_{i_{j}}|_{\times}\Big)^{-1}  \leq \prod\limits^{n}_{j=m} a_{i_{j}} \leq \prod\limits^{n}_{j=m} |a_{i_{j}}|_{\times}
\end{eqnarray*}
give the convergence of $\prod\limits^{\infty}_{j=1} a_{i_{j}}$, by Theorem 1.3. Similarly, the inequalities, for $n>m$,
\begin{eqnarray*}
\Big(\prod \limits^{n}_{i=m} \Big(\prod\limits^{\infty}_{j=1} |a_{i_{j}}|_{\times}\Big)\Big)^{-1} \leq \prod\limits^{n}_{i=m} \prod\limits^{\infty}_{j=1} a_{i_{j}} \leq \prod\limits^{n}_{i=m} \prod\limits^{\infty}_{j=1} |a_{i_{j}}|_{\times}
\end{eqnarray*}
give the convergence of $\prod\limits^{\infty}_{i=1} \prod\limits^{\infty}_{j=1} a_{i_{j}}$.

\ To prove the converse part, assume that $\prod\limits_{i\in I} |a_{i}|_{\times}$ does not converge, or equivalently
$\prod\limits^{\infty}_{n=1} |a_{n}|_{\times}$ does not converge, when $I$ is considered in the form $\{1,2,...\}=I_{1}\cup I_{2}$ as stated in the statement for (e). Then either $\prod\limits_{i\in I_{1}} |a_{i}|_{\times}=+\infty$ or $\prod\limits_{i\in I_{2}} |a_{i}|_{\times}= +\infty$. For example, let us consider the case $\prod\limits_{i\in I_{1}} |a_{i}|_{\times}=+\infty$. In this case, let us find an infinite subset $J_{1}= \{1_{1},1_{2},...\}$ of $I$ such that $I_{1}\setminus J_{1}$ is infinite and such that $\prod\limits^{\infty}_{j=1} a_{1_{j}}$ does not converge. This is possible in view of Remark 3.2 and Remark 1.4. Write $(I_{1} \setminus J_{1})\cup I_{2}$ in the form $J_{2}\cup J_{2}\cup...$ such that each $J_{i}=\{i_{1},i_{2},...\}$ is an infinite set and such that $J_{i}$ are pairwise disjoint. Then $\prod\limits^{\infty}_{i=1} \prod\limits^{\infty}_{j=1} a_{i_{j}}$ does not converge.

\ It remains to establish the final part of (e). Let $p=\prod\limits^{\infty}_{i=1} \prod\limits^{\infty}_{j=1} a_{i_{j}}$
for some partition $I=J_{1}\cup J_{2}\cup...$. For each $n$, it is possible to find integers $m(n),n_{1},n_{2},...,n_{m(n)}$ such that $\Big|\prod\limits^{m(n)}_{i=1} \prod\limits^{n_{i}}_{j=1} a_{i_{j}}-p\Big|_{+} < \frac{1}{n}$, and such that $k_{1>}n_{1}, k_{2}>n_{2},...,k_{m(n)}>n_{m(n)}$, $m(k)> m(n)$, whenever $k>n$. Let $D=\Big\{ \{i_{j}: 1 \leq j \leq n_{i}, 1 \leq i \leq m(n) \}:n=1,2,...\Big\}$. Consider $D$ as a directed set under inclusion relation, and when $I$ is considered as $\{i_{j}: j=1,2,... , i=1,2,...\}$. Now the final part of (d) implies that $p=\prod\limits_{i\in I}a_{i}$. The proof is now completed.
\end{proof}
\begin{rem}
\ In the previous Theorem 4.4, (e) includes iterated convergence for double products. The part (d) includes some special types of net convergence meant for double products corresponding to the one known for double series in general forms. References for such general forms may be found in [3]. By Theorem 3.3 and Remark 3.2, it can be stated that all parts (a), (b), (c), (d),(e) are equivalent to one more part. (f): Consider $I$ in the form $\{1,2,...\}$ (as it was done in (e)). Then all rearrangements $\prod\limits^{\infty}_{n=1} a_{n}$ of $\prod\limits_{i\in I}a_{i}$ converge. Moreover, they all converge to $\prod\limits_{i\in I}a_{i}$.
\end{rem}

\section{Tests For Convergence}
\ If a necessary condition for convergence is not satisfied then convergence fails. If a sufficient condition for convergence is satisfied then convergence is assured. All results giving conditions for convergence are classified under tests for convergence. For example, if $\limsup\limits_{n \rightarrow \infty} |\log a_{n}|_{+}^{\frac{1}{n}} <1$ then $\prod\limits^{\infty}_{n=1} a_{n}$ converges, formally, according to the root test for series convergence. This statement can be made into a logically correct statement. But, it is not our aim. Our aim is to use necessary transformations in arguments of proofs of results. Let us consider again countable infinite products.
\begin{prop}
\ Let $(t_{n})_{n=1}^{\infty}$ and $(a_{n})_{n=1}^{\infty}$ be sequences of positive numbers such that $\sum\limits^{\infty}_{n=1} t_{n}$ converges and such that $\limsup\limits_{n\rightarrow \infty}a_{n}^{\frac{1}{t_{n}}}=\beta <\infty$ and $\liminf\limits_{n\rightarrow \infty}a_{n}^{\frac{1}{t_{n}}}=\alpha >0$. Then $\prod\limits^{\infty}_{n=1} a_{n}$ converges.
\end{prop}

\begin{proof}
\ Fix $\alpha^{'}$ and $\beta^{'}$ such that $0< \alpha^{'} <\alpha \leq \beta <\beta^{'} <\infty$. Then there is an integer $n_{0}$ such that $\alpha^{'} < a_{n}^{\frac{1}{t_{n}}} <\beta^{'}$, $\forall$ $n \geq n_{0}$. Then, for $n>m \geq n_{0}$,
\begin{eqnarray*}
 {\alpha^{'}}^{\sum\limits^{n}_{k=m} t_{k}} < \prod\limits^{n}_{k=m} a_{n} <{\beta^{'}}^{\sum\limits^{n}_{k=m} t_{k}}.
\end{eqnarray*}
Now, the proposition follows from Theorem 1.3.
\end{proof}

\begin{cor}
\ Let $(t_{n})_{n=1}^{\infty}$ be a sequence of positive numbers such that $\sum\limits^{\infty}_{n=1} t_{n}$ converges. Let $(a_{n})_{n=1}^{\infty}$ be a sequence of positive numbers such that $a_{n} \geq 1 $, $\forall n$, and such that $\limsup\limits_{n\rightarrow \infty}a_{n}^{\frac{1}{t_{n}}} <\infty$. Then $\prod\limits^{\infty}_{n=1} a_{n}$ converges. Here $``a_{n} \geq 1"$ and $``\limsup\limits_{n\rightarrow \infty}a_{n}^{\frac{1}{t_{n}}} <\infty"$ can be replaced by $``a_{n} \leq 1"$ and $``\liminf\limits_{n\rightarrow \infty}a_{n}^{\frac{1}{t_{n}}} >0"$ to get the same conclusion: $\prod\limits^{\infty}_{n=1} a_{n}$ converges.
\end{cor}
\begin{thm}
\ Suppose $a_{1}\geq a_{2}\geq... \geq 1$.Then $\prod\limits^{\infty}_{n=1} a_{n}$ converges if and only if $\prod\limits^{\infty}_{k=0} a_{2^{k}}^{2^{k}}= a_{1}^{1}a_{2}^{2}a_{4}^{4}a_{8}^{8}....$ converges.
\end{thm}
\begin{proof}
\ Let $u_{n}=a_{1} a_{2}...a_{n}$, and  $v_{k}=a_{2^{0}}^{2^{0}} a_{2^{1}}^{2^{1}} a_{2^{2}}^{2^{2}}....a_{2^{k}}^{2^{k}}$. For $n<2^{k}$,
\begin{eqnarray*}
 u_{n} \leq a_{1}(a_{2}a_{3})...(a_{2^{k}}a_{2^{k}+1}...a_{2^{k+1}-1}) \leq a_{1}^{1}a_{2}^{2}a_{4}^{4}...a_{2^{k}}^{2^{k}}=v_{k}
\end{eqnarray*}
so that $u_{n} \leq v_{k}$. On the other hand, if $n>2^{k}$,
\begin{eqnarray*}
 u_{n} \geq a_{1}^{\frac{1}{2}}a_{2}(a_{3}a_{4})...(a_{2^{k-1}-1}...a_{2^{k}}) \geq (a_{1}^{1}a_{2}^{2}a_{4}^{4}...a_{2^{k}}^{2^{k}})^{\frac{1}{2}}=v_{k}^{\frac{1}{2}}
\end{eqnarray*}
so that $u_{n}^{2} \geq v_{k}$. These two inequalities establish the result, because $a_{n} \geq 1$, $\forall n$.
\end{proof}
\begin{defn}
\ Suppose $a_{n}>1$, for each $n$. The product $a_{1}a_{2}^{-1}a_{3}a_{4}^{-1}a_{5}a_{6}^{-1}$.... is called an alternating product.
\end{defn}

\begin{thm}
\ If $(a_{n})_{n=1}^{\infty}$ is a strictly decreasing sequence converging to 1, then the alternating product $a_{1}a_{2}^{-1}a_{3}a_{4}^{-1}a_{5}a_{6}^{-1}$... converges.
\end{thm}
\begin{proof}
\ Let $b_{1}= a_{1}a_{2}^{-1}, b_{2}=a_{3}a_{4}^{-1}, b_{3}=a_{5}a_{6}^{-1}$,... .Then $b_{n}>1$, $\forall n$. Hence, $\prod\limits^{\infty}_{n=1} b_{n}$ converges if and only if $(\prod\limits^{n}_{k=1} b_{k})^{\infty}_{n=1}$ is a bounded sequence. Let us observe that $\prod\limits^{n}_{k=1} b_{k}= a_{1}(\frac{a_{2}}{a_{3}})^{-1} (\frac{a_{4}}{a_{5}})^{-1} ...$ $ (\frac{a_{2n-2}}{a_{2n-1}})^{-1}a_{2n}^{-1}<a_{1}$, because $(\frac{a_{2}}{a_{3}})^{-1}, (\frac{a_{4}}{a_{5}})^{-1}, ... (\frac{a_{2n-1}}{a_{2n-1}})^{-1},a_{2n}^{-1}$ are less than 1. So, $\prod\limits^{\infty}_{n=1} b_{n}$ converges. The value of $a_{1}a_{2}^{-1}a_{3}a_{4}^{-1}a_{5}a_{6}^{-1}...a_{m}^{(-1)^{m+1}}$ is $\prod\limits^{\frac{m}{2}}_{k=1} b_{k}$ if $m$ is even, and it is $\Big(\prod\limits^{\frac{m-1}{2}}_{k=1} b_{k}\Big)a_{m}$ if $m$ is odd. So, $a_{1}a_{2}^{-1}a_{3}a_{4}^{-1}$... converges to $\prod\limits^{\infty}_{n=1} b_{n}$, because $a_{m} \rightarrow 1$ as $m\rightarrow \infty$.
\end{proof}

\begin{cor}
\ If $(a_{n})_{n=1}^{\infty}$ is a strictly decreasing sequence in $(0,1)$ which converges to zero, then $(1-a_{1})(1-a_{2})^{-1}(1-a_{3})(1-a_{4})^{-1}$... converges.
\end{cor}
\begin{proof}
\ Here, $((1-a_{n})^{-1})_{n=1}^{\infty}$ is a strictly decreasing sequence, which converges to 1.
\end{proof}

\begin{thm}
\ Let $(a_{n})_{n=1}^{\infty}$ be a sequence of positive numbers such that $\prod\limits^{\infty}_{n=1} a_{n}$ converges to $p$. Let $(t_{n})_{n=1}^{\infty}$ be a sequence of positive numbers such that $t_{1}+t_{2}+... =+\infty$. Let $u_{n}=\prod\limits^{n}_{k=1} a_{k}$, and $\sigma_{n}= (u_{1}^{t_{1}} u_{2}^{t_{2}}...u_{n}^{t_{n}})^{\frac{1}{t_{1}+t_{2}+...+t_{n}}}$, $\forall n$. Then $\sigma_{n}$ converges to $p$ as $n \rightarrow \infty$.
\end{thm}

\begin{proof}
\ Let $v_{n}=\frac{u_{n}}{p}$ and $\tau_{n}=\frac{\sigma_{n}}{p}$, $\forall n$. Then $(v_{1}^{t_{1}} v_{2}^{t_{2}}...v_{n}^{t_{n}})^{\frac{1}{t_{1}+t_{2}+...+t_{n}}}=\frac{\sigma_{n}}{p}$, $\forall n$. Choose a finite positive number A such that $|v_{n}|_{\times} \leq A$, $\forall n$. Given $\epsilon >0$, choose an integer $k$ such that $|v_{n}|_{\times} \leq 1+\epsilon$, $\forall n>k$. Then
\begin{eqnarray*}
1 \leq |\tau_{n}|_{\times} = \Big| \frac{\sigma_{n}}{p} \Big|_{\times} \leq  (|v_{1}|_{\times}^{t_{1}} |v_{2}|_{\times}^{t_{2}}...|v_{k}|_{\times}^{t_{k}})^{\frac{1}{t_{1}+t_{2}+...+t_{n}}} (|v_{k+1}|_{\times}^{t_{k+1}}...|v_{n}|_{\times}^{t_{n}})^{\frac{1}{t_{1}+t_{2}+...+t_{n}}}\\
 \leq A^{\frac{t_{1}+t_{2}+...+t_{k}}{t_{1}+t_{2}+...+t_{n}}} (1+\epsilon)^{\frac{t_{k+1}+t_{2}+...+t_{n}}{t_{1}+t_{2}+...+t_{n}}}.
\end{eqnarray*}
\ Thus $1 \leq |\tau_{n}|_{\times} \leq 1+2\epsilon$ for sufficiently large $n$, because $ t_{1}+t_{2}+...=  \infty$. This proves that $ \lim\limits_{n \rightarrow\infty} |\tau_{n}|_{\times}=1= \lim\limits_{n \rightarrow\infty} \tau_{n}$. So, $ \lim\limits_{n \rightarrow \infty} \sigma_{n}=p. $
\end{proof}
\begin{cor}
\ Let $(a_{n})_{n=1}^{\infty}$ be a sequence of positive numbers such that $\prod\limits^{\infty}_{n=1} a_{n}$ converges to $p$. Let $u_{n}=\prod\limits^{n}_{k=1} a_{k}$, and $\sigma_{n}= (u_{1} u_{2}...u_{n})^{\frac{1}{n}}$, $\forall n$. Then $\sigma_{n}$ converges to $p$ as $n \rightarrow \infty$.
\end{cor}
\begin{thm}
\ Suppose $(b_{n})_{n=1}^{\infty}$ be a sequence of positive reals such that $\prod\limits^{\infty}_{n=1} \frac{b_{n}}{b_{n+1}}$ converges m-absolutely and such that $b_{n} \rightarrow 1$ as $n\rightarrow\infty$. Suppose $(a_{n})_{n=1}^{\infty}$ be a sequence of real numbers such that $\Big\{\sum\limits^{m}_{k=n}a_{k}: m=1,2,..., n=1,2,..., n \leq m \Big\}$ is a bounded set. Then $\prod\limits^{\infty}_{k=1}b_{k}^{a_{k}}$ converges.
\end{thm}
\begin{proof}
\ Suppose $\Big|\sum\limits^{m}_{k=n}a_{k}\Big|_{+} \leq M$, $\forall$ $n=1,2,...$, $\forall m=1,2,...,$ subject to the condition $n \leq m$, for some $M>0$. Note that, for $n>2$,
\begin{eqnarray*}
 b_{1}^{a_{1}} b_{2}^{a_{2}}...b_{n}^{a_{n}}=\Big(\frac{b_{1}}{b_{2}}\Big)^{a_{1}} \Big(\frac{b_{2}}{b_{3}}\Big)^{a_{1}+a_{2}}...\Big(\frac{b_{n-1}}{b_{n}}\Big)^{a_{1}+a_{2}+...+a_{n-1}} b_{n}^{a_{1}+a_{2}+...+a_{n}}
\end{eqnarray*}
Then
\begin{eqnarray*}
 \Big(\Big|\frac{b_{1}}{b_{2}}\Big|_{\times} \Big|\frac{b_{2}}{b_{3}}\Big|_{\times}...\Big|\frac{b_{n-1}}{b_{n}}\Big|_{\times} \Big)^{-M} |b_{n}|_{\times}^{-M} \leq b_{1}^{a_{1}} b_{2}^{a_{2}}...b_{n}^{a_{n}} \leq \Big(\Big|\frac{b_{1}}{b_{2}}\Big|_{\times} \Big|\frac{b_{2}}{b_{3}}\Big|_{\times}...\Big|\frac{b_{n-1}}{b_{n}}\Big|_{\times}\Big)^{M} |b_{n}|_{\times}^{M}.
\end{eqnarray*}
\ One can derive similar inequalities for $b_{m}^{a_{m}} b_{m+1}^{a_{m+1}}...b_{n}^{a_{n}}$, and for $m<n$, and then Theorem 1.3 can be applied to derive convergence of $\prod\limits^{\infty}_{k=1} {b_{k}}^{a_{k}}$.
\end{proof}
\ This Theorem 5.9 gives a motivation for concepts and results to be given in Sections 6 and 7. Theorem 5.7 and Theorem 5.9 of this Section 5 are comparable with the corresponding classical summability results, Theorem 8.48 in \cite{a1} and Theorem 8.27 in \cite{a1}, which attributed to the mathematicians Cesaro and Abel, respectively.
\section{Matrices and Multiplicability}
\ Let $(0,\infty)^{n}$ denote the cartesian product $\prod\limits^{n}_{i=1} X_{i}$ in which each $X_{i}=(0,\infty)$, multiplication is defined coordinatewisely, and elements are written as transpose of row vectors $(x_{1},x_{2},...,x_{n})$. More explicitly,  \\
\begin{equation*}\label{eq:matrixeqn}
 \begin{pmatrix}
 x_{1} \\
 x_{2} \\
 . \\
 . \\
 . \\
 x_{n} \\
 \end{pmatrix}
 \begin{pmatrix}
 y_{1} \\
 y_{2} \\
 . \\
 . \\
 . \\
 y_{n} \\
 \end{pmatrix}
 =
 \begin{pmatrix}
 x_{1}y_{1} \\
 x_{2}y_{2} \\
 . \\
 . \\
 . \\
 x_{n}y_{n} \\
 \end{pmatrix}
 \end{equation*}
or
$(x_{1},x_{2},...,x_{n})^{T} (y_{1},y_{2},...,y_{n})^{T}=(x_{1}y_{1},x_{2}y_{2},...,x_{n}y_{n})^{T}$ in $(0,\infty)^{n}$. Let $A=(a_{ij})$ be a matrix of order $m\times n$ with real entries $a_{ij}$. Let us define \\
\begin{equation*}\label{eq:matrixeqn}
 A*
 \begin{pmatrix}
 x_{1} \\
 x_{2} \\
 . \\
 . \\
 . \\
 x_{n} \\
 \end{pmatrix}
 =
 \begin{pmatrix}
 x_{1}^{a_{11}}x_{2}^{a_{12}}...x_{n}^{a_{1n}} \\
 x_{1}^{a_{21}}x_{2}^{a_{22}}...x_{n}^{a_{2n}} \\
 . \\
 . \\
 . \\
 x_{1}^{a_{m1}}x_{2}^{a_{m1}}...x_{n}^{a_{mn}}\\
 \end{pmatrix}
 \end{equation*}
 such that A is a multiplication preserving function from $(0,\infty)^{n}$ to $(0,\infty)^{m}$.  Then, for a given matrix $B=(b_{ij})$ of order $k\times m$ with real entries, it can be verified that \\
 \begin{equation*}\label{eq:matrixeqn}
 B*
 \begin{pmatrix}
 A*
 \begin{pmatrix}
 x_{1} \\
 x_{2} \\
 . \\
 . \\
 . \\
 x_{n} \\
 \end{pmatrix}
 \end{pmatrix}
 =
 (BA)*
 \begin{pmatrix}
 x_{1} \\
 x_{2} \\
 . \\
 . \\
 . \\
 x_{n} \\
 \end{pmatrix},
 \end{equation*} \\
 $\forall$ $(x_{1},x_{2},...,x_{n})^{T} \in (0,\infty)^{n}$, where $BA$ is the usual matrix multiplication. Let $C$ be another matrix of order $k\times m$ with real entries. Then the following is true for every $(x_{1},x_{2},...,x_{n})^{T} \in (0,\infty)^{n}$:\\
  \begin{equation*}\label{eq:matrixeqn}
 (C+B)*
 \begin{pmatrix}
 A*
 \begin{pmatrix}
 x_{1} \\
 x_{2} \\
 . \\
 . \\
 . \\
 x_{n} \\
 \end{pmatrix}
 \end{pmatrix}
 =
 \begin{pmatrix}
 C*
 \begin{pmatrix}
 A*
 \begin{pmatrix}
 x_{1} \\
 x_{2} \\
 . \\
 . \\
 . \\
 x_{n} \\
 \end{pmatrix}
 \end{pmatrix}
 \end{pmatrix}
 \begin{pmatrix}
 B*
 \begin{pmatrix}
 A*
 \begin{pmatrix}
 x_{1} \\
 x_{2} \\
 . \\
 . \\
 . \\
 x_{n} \\
 \end{pmatrix}
 \end{pmatrix}
 \end{pmatrix}.
 \end{equation*}
 Also,
 \begin{equation*}\label{eq:matrixeqn}
 (C+B)*
 \begin{pmatrix}
 y_{1} \\
 y_{2} \\
 . \\
 . \\
 . \\
 y_{m} \\
 \end{pmatrix}
 =
 \begin{pmatrix}
 C*
 \begin{pmatrix}
 y_{1} \\
 y_{2} \\
 . \\
 . \\
 . \\
 y_{m} \\
 \end{pmatrix}
 \end{pmatrix}
 \begin{pmatrix}
 B*
 \begin{pmatrix}
 y_{1} \\
 y_{2} \\
 . \\
 . \\
 . \\
 y_{m} \\
 \end{pmatrix}
 \end{pmatrix},
 \end{equation*}\\
 $\forall$ $(y_{1},y_{2},...,y_{m})^{T} \in (0,\infty)^{m}$. Let $D$ and $E$ be matrices of order $n\times k$ with real entries. Then it can be verified that
 \begin{equation*}\label{eq:matrixeqn}
 A*
 \begin{pmatrix}
 (D+E)*
 \begin{pmatrix}
 z_{1} \\
 z_{2} \\
 . \\
 . \\
 . \\
 z_{k} \\
 \end{pmatrix}
 \end{pmatrix}
 =
 \begin{pmatrix}
 A*
 \begin{pmatrix}
 D*
 \begin{pmatrix}
 z_{1} \\
 z_{2} \\
 . \\
 . \\
 . \\
 z_{k} \\
 \end{pmatrix}
 \end{pmatrix}
 \end{pmatrix}
 \begin{pmatrix}
 A*
 \begin{pmatrix}
 E*
 \begin{pmatrix}
 z_{1} \\
 z_{2} \\
 . \\
 . \\
 . \\
 z_{k} \\
 \end{pmatrix}
 \end{pmatrix}
 \end{pmatrix}
 \end{equation*}\\
 $\forall$ $(z_{1},z_{2},...,z_{k})^{T} \in (0,\infty)^{k}$.
 \ Thus, there are interesting matrix operations. With this familiarity of the operation $*$ for finite matrices, let us extend the same for infinite matrices or double sequences which are associated with summability. Corollary 5.8 is also a counterpart of a summability method.
\begin{prop}
\ Let $(a_{m,n})_{m,n=1}^{\infty}$ be a double sequence of positive numbers such that $\sum\limits^{\infty}_{n=1} a_{m,n} \leq M$, for every $m$, for some $M>0$, and such that $\lim \limits_{m\rightarrow \infty}a_{m,n}=0$, for every $n$. Then, for every sequence $(x_{n})_{n=1}^{\infty}$ of positive reals for which $\lim \limits_{n\rightarrow \infty}x_{n}=p>0$ exists, $\lim \limits_{m\rightarrow \infty} \prod\limits^{\infty}_{n=1} (\frac{x_{n}}{p})^{a_{m,n}}=1$. Equivalently, $\lim \limits_{m\rightarrow \infty} \prod\limits^{\infty}_{n=1} {x_{n}}^{a_{m,n}}=1$, whenever $\lim \limits_{n\rightarrow \infty}x_{n}=1$ and $x_{n}>0$, $\forall n$.
\end{prop}
\begin{proof}
\ Suppose $(x_{n})_{n=1}^{\infty}$ be a sequence of positive reals such that $x_{n}\rightarrow 1$ as $n \rightarrow\infty$. Fix $\epsilon >0$. Find an integer $n_{0}$ such that $|x_{k}|_{\times}<(1+\epsilon)^{\frac{1}{2M}}$, $\forall$ $k \geq n_{0}$. Then, for $l>k \geq n_{0}$,
\begin{eqnarray*}
(1+\epsilon)^{\frac{-1}{2}} < \prod \limits^{l}_{s=k} |x_{s}|_{\times}^{-a_{m,s}} \leq \prod\limits^{l}_{s=k} {x_{s}}^{a_{m,s}} \leq \prod \limits^{l}_{s=k} |x_{s}|_{\times}^{a_{m,s}}< (1+\epsilon)^{\frac{1}{2}},
\end{eqnarray*}
$\forall$ $m$, and hence
\begin{eqnarray*}
(1+\epsilon)^{\frac{-1}{2}} \leq \prod\limits^{\infty}_{n=n_{0}+1} {x_{n}}^{a_{m,n}} \leq (1+\epsilon)^{\frac{1}{2}},
\end{eqnarray*}
$\forall$ $m$. Find an integer $m_{0}$ such that
\begin{eqnarray*}
(1+\epsilon)^{\frac{-1}{2}} < \prod \limits^{n_{0}}_{k=1} |x_{k}|_{\times}^{-a_{m,k}} \leq \prod\limits^{n_{0}}_{k=1} {x_{k}}^{a_{m,k}} \leq \prod \limits^{n_{0}}_{k=1} |x_{k}|_{\times}^{a_{m,k}}< (1+\epsilon)^{\frac{1}{2}},
\end{eqnarray*}
$\forall$ $m \geq m_{0}$. Thus, for $m \geq m_{0}$,
\begin{eqnarray*}
(1+\epsilon)^{-1} \leq \prod\limits^{n_{0}}_{n=1} {x_{n}}^{a_{m,n}} \prod\limits^{\infty}_{n=n_{0}+1} {x_{n}}^{a_{m,n}} \leq (1+\epsilon).
\end{eqnarray*}
That is, for given $\epsilon >0$, there is an integer $m_{0}$ such that
\begin{eqnarray*}
(1+\epsilon)^{-1} \leq \prod\limits^{\infty}_{n=1} {x_{n}}^{a_{m,n}} \leq (1+\epsilon),
\end{eqnarray*}
$\forall$ $m \geq m_{0}$. This proves that $\lim \limits_{m\rightarrow \infty} \prod\limits^{\infty}_{n=1} {x_{n}}^{a_{m,n}}=1$.
\end{proof}
\section{Power Products}
\ Corresponding to power series, it is possible to introduce power products.
\begin{defn}
\ Let $(a_{n})_{n=1}^{\infty}$ be a sequence of positive reals. A formal power product is $a_{0}{a_{1}}^{x}{a_{2}}^{x^{2}}{a_{3}}^{x^{3}}$... or $\prod\limits^{\infty}_{n=0} {a_{n}}^{x^{n}}$.
\end{defn}
\begin{rem}
\ This formal power product converges to a value, if $|x|_{+}< \limsup \limits_{n \rightarrow\infty} |\log a_{n}|_{+}^{\frac{1}{n}}$. This can be derived from the corresponding power series obtained by logarithmic transformation. This is not our method of deriving results. Let us now recall Theorem 5.9 in another version.
\end{rem}
\begin{thm}
\ Suppose $(b_{n})_{n=1}^{\infty}$ be as in Theorem 5.9. Then the power product $\prod\limits^{\infty}_{k=1} {b_{k}}^{x^{k}}$ converges for every real $x \in [-1,1)$.
\begin{proof}
\ Let $x \in [-1,1)$. Write $a_{k}=x^{k}$, for $k=1,2,...$. Then $\Big\{\sum\limits^{m}_{k=n}a_{k}: m=1,2,..., n=1,2,..., n \leq m \Big\}$ is a bounded set. By Theorem 5.9, $\prod\limits^{\infty}_{k=1} {b_{k}}^{x^{k}}$ converges.
\end{proof}
\end{thm}
\begin{cor}
\ Let $(b_{n})_{n=1}^{\infty}$ be a strictly increasing sequence of positive reals converging to 1. Then the power product $\prod\limits^{\infty}_{k=1} {b_{k}}^{x^{k}}$ converges for every real $x \in [-1,1)$.
\end{cor}
\begin{proof}
\ For $n>m$,
\begin{eqnarray*}
1 & \leq & \Big|\frac{b_{m}}{b_{m+1}}\Big|_{\times} \Big|\frac{b_{m+1}}{b_{m+2}}\Big|_{\times}...\Big|\frac{b_{n}}{b_{n+1}}\Big|_{\times} \\
&=& \frac{b_{m+1}}{b_{m}}\frac{b_{m+2}}{b_{m+1}}...\frac{b_{n+1}}{b_{n}}= \frac{b_{n+1}}{b_{m}}< \frac{1}{b_{m}},
\end{eqnarray*}
and $\frac{1}{b_{m}}\rightarrow 1$ as $m\rightarrow \infty$. Therefore, by Theorem 1.3, $\prod\limits^{\infty}_{n=1}|\frac{b_{n}}{b_{n+1}}|_{\times}$ converges. Now, the conclusion follows from Theorem 7.3.
\end{proof}
\ So, there is a need to improve Theorem 5.9. The next result may remind Cauchy product for sequences.
\begin{thm}
\ Let $(a_{n})_{n=1}^{\infty}$ and $(b_{n})_{n=1}^{\infty}$ be two sequences of positive reals such that $\prod\limits^{\infty}_{n=0} a_{n}$ and $\prod\limits^{\infty}_{n=0} b_{n}$ converge to $A$ and $B$. For each $n$, let $ c_{n}=a_{0}b_{n}a_{1}b_{n-1}...a_{n}b_{0}$, $A_{n}=\prod\limits^{n}_{k=0} a_{k}$, $B_{n}=\prod\limits^{n}_{k=0} b_{k}$ and $C_{n}=\prod\limits^{n}_{k=0} c_{k}$. Then ${A_{n}}^{\frac{1}{n}} \rightarrow 1$, ${B_{n}}^{\frac{1}{n}} \rightarrow 1$, ${C_{n}}^{\frac{1}{n}} \rightarrow AB$ and ${c_{n}}^{\frac{1}{n}} \rightarrow 1$ as $n\rightarrow\infty$.
\end{thm}
\begin{proof}
\ By Theorem 1.3, $a_{n} \rightarrow 1$, $b_{n} \rightarrow 1$ as $n\rightarrow\infty$. By Lemma 2.5, all conclusions follow, because $C_{n}=A_{0}A_{1}...A_{n}B_{0}B_{1}...B_{n}$, and $A_{n} \rightarrow A$, $B_{n} \rightarrow B$ as $n\rightarrow\infty$.
\end{proof}
\section{Conclusion}
\ The problem in using transformation technique lies in guessing methods for transformation. For example, transformations were not used since nineteenth century to define multiplicative modulus function from additive modulus function. One can guess that a transformation can be applied to derive multiplicative modulus function from additive modulus function, only  after introducing the concept of multiplicative modulus function. So, separate techniques should be developed for infinite products. Theory of infinite products is also applicable like theory of infinite series. This is the need for development of theory of infinite products. Corollary 7.4 is simple to write examples. But it is not sufficient. So, Theorem 5.9 should be improved in all possible ways, but subject to simple conditions.

\textbf{Declaration:} There is no conflict of interest.

\subsection*{Acknowledgement}
Dr. C. Ganesa Moorthy (Professor, Department of Mathematics, Alagappa University, Karaikudi- 630003, INDIA) gratefully acknowledges the joint financial support of RUSA-Phase 2.0 grant sanctioned vide letter No.F 24-51/2014-U, Policy (TN Multi-Gen), Dept. of Edn. Govt. of India, Dt. 09.10.2018, UGC-SAP (DRS-I) vide letter No.F.510/8/DRS-I/2016 (SAP-I) Dt. 23.08.2016 and DST (FIST - level I) 657876570 vide letter No.SR/FIST/MS-I/2018-17 Dt. 20.12.2018.


\begin{thebibliography}{10}
\normalsize
\baselineskip=30pt

\bibitem{a1}
 T. M. Apostol, {\it Mathematical Analysis}, 2nd ed., Addison-Wesley, Reading, 1974.
\bibitem{k1}
J. L. Kelly, {\it General Topology}, Van Nostrand, New York, 1955.
\bibitem{l1}
S. Loganathan and C. G. Moorthy, {\it A net convergence for Schauder double bases}, Asian-European J. Math., \textbf{9}(1) (2016), 1-33.
\end{thebibliography}
\end{document}